\documentclass[sn-mathphys]{sn-jnl_a}

\theoremstyle{thmstyleone}%
\newtheorem{theorem}{Theorem}
\newtheorem{proposition}[theorem]{Proposition}%
\newtheorem{corollary}[theorem]{Corollary}
\newtheorem{lemma}[theorem]{Lemma}

\theoremstyle{thmstyletwo}%
\newtheorem{example}{Example}%
\newtheorem{remark}{Remark}%

\theoremstyle{thmstylethree}%
\newtheorem{definition}{Definition}%

\newcommand{\mc}{\mathcal}

\newcommand{\wh}{\widehat}
\newcommand{\ol}{\overline}
\newcommand{\wt}{\widetilde}

\begin{document}

\title[Neumann property in the extended modular group]{Neumann property in the extended modular group and maximal nonparabolic subgroups of the modular group}


\author{\fnm{Andrzej} \sur{Matra\'{s}}}\email{matras@uwm.edu.pl}
\author[1]{\fnm{Artur} \sur{Siemaszko}}\email{artur@uwm.edu.pl}


\affil[1]{\orgdiv{Faculty of Mathematics and Computer Science}, \orgname{University of Warmia and Mazury in Olsztyn}, \orgaddress{\street{Sloneczna 54 Street}, \postcode{10-710} \city{Olsztyn}, \country{Poland}}}




\abstract{
It is know that any Neumann subgroup of the modular group is maximal non-parabolic. The question arises as to whether these are the only maximal non-parabolic subgroups. The wild class of maximal, non-parabolic, not Neumann subgroups of the modular group was constructed by Brenner and Lyndon.  The new construction of such a class is presented.  Those groups are obtained as subgroups of elements of positive determinant of any Neumann subgroup of the extended modular group (the notion which we introduce in the paper) and in this sens they are as "close"  as possible to Neumann subgroups of the modular group.
}




\maketitle

\section{Introduction} The inhomogeneous modular group $PSL(2,Z)$, denoted in the paper by $\mc M$, acts on the upper half plane $\mathbb{H}$ via $g(z)=\frac{az+b}{cz+d}$ for $g=\pm\mbox{{\small $\left(\begin{array}{cc}a &b\\c&d\end{array}\right)$}}\in \mc{M}$. This group is contained as a normal subgroup of index $2$ in the extended modular group $PGL(2,Z)$ (denoted by $\wh{\mc M}$) which acts on $\mathbb{H}$ as $g(r\cdot e^{i\varphi})=\frac {a\cdot{r\cdot e^{i\varphi\cdot det\,g}}+b}{c\cdot{r\cdot e^{i\varphi\cdot det\,g}} +d}$ for $g\in {\wh{\mc M}}$. We will use the following presentation  of the extended modular group:
\begin{equation}\label{presemg}
\wh{\mc M}=\left<\omega, \varphi, \nu \;\mid\; \omega^2 = \varphi^3=\nu^2 = (\omega \nu)^2 = (\omega \varphi \nu )^2  = 1 \right>,
\end{equation}
where $\omega(z) = -\frac{1}{z}$,  $\varphi (z)=\frac{-1}{z+1}$ and $\nu(z) = -z$. Using this generators we have the following presentation of the modular group:
$$ \mc M=\left< \omega, \varphi \;\mid\; \omega^2 = \varphi^3  =1 \right>.$$
\indent Parabolic elements of $\mc M$ are conjugate to elements of the form $\tau^n$ ($n \in Z$), where $\tau=\omega\varphi$. Any subgroup $\mc M$ of a finite index contains parabolic elements, however there are nonparabolic subgroups of $\mc M$ of infinite index.\\
Neumann (\cite{N}) investigated some class of subgroups of the homogeneous modular
group $SL(2, Z)$.  As a motivation, he
refers to the work of Schmidt on the foundations of geometry, in which a
construction of some subgroup of $SL(2, Z)$ containing $\omega$ is needed.
He as the first explicitly constructed  the  continuum of such distinct groups.  Subgroups of $PSL(2,Z)$ which are complement to a subgroup conjugate to $$\mc T = \left\{\tau^n :\; n \in Z \right\}$$ are precisely projections of the groups defined in \cite{N}. Magnus in  \cite{Ma} called them \emph{ Neumann subgroups} of the modular group. In the paper \cite{T} Tretkoff  constructed Neumann subgroups with all possible structures subject to the condition that the free product decomposition  contains a group of order two and eventually a free group of even rank.  As it was proved by Stothers in \cite{S} (see also \cite{BL1}) the  odd rank of a free part of the decomposition has to be excluded.\\
Nonparabolic subgroup appear as well in other contexts. For instance,  if $F$ is a normal, free subgroup of a finite index in $SL(2,Z)$ then the commutant of $F$ is nonparabolic \cite{MD}. Also they are of interest from point of view of ergodic theory, because they are precisely groups that act in a mixing fashion on the torus $\mathbb T^{2}$ (\cite{BE}).\\
In the series of papers  we investigated the Cayley property of the distant graph (see \cite{He}) for the definition) of the projective line over $Z$. There is a natural way to generalize the notion of  being a Neumann subgroup to the extended modular group. In \cite{MS3}  we finally  showed that such defined class coincide with the set of  Cayley representations of the distant graph mentioned above, previously studied in \cite{MS1} and \cite{MS2}.

It is known that every Neumann subgroup of the modular group is a maximal nonparabolic one. Magnus (\cite{Ma}) conjectured that Neumann groups are the only maximal nonparabolic. The answer is negative.  Brennen and Lyndon in \cite{BL2} gave the first construction using symmetries of the regular tessellation of the Euclidean plane by hexagons and get subgroups with arbitrary even number orbits of $\mc T$ in their coset graphs. Jones, in the paper \cite{J1}, constructed nonparabolic and maximal subgroups in the triangle group $\Gamma \simeq C_p \star C_q$ ($p \geq 2$, $q \geq 3$). Neumann subgroups and their maximality have also been studied geometrically by Kulkarni in (\cite{K1}, \cite{K2}). He takes $\Gamma$ to be a free product of an arbitrary finite number of finite cyclic groups of prime order (\cite{K1}). Some generalisations of results connected with maximal nonparabolic groups can be find in (\cite{J2}).

  In the paper we present the new construction of a wide class of maximal nonparabolic subgroups of  $\mc M$, which are not Neumann groups.  We get them in a completely different way than Brennen and Lyndon.
  First (analogously as in the case of $\mc M$) we introduce the notion of a Neumann subgroup of the extended modular  as a maximal completion of a subgroup conjugate to a maximal isotropic subgroup $$\wh{\mc T}=\{\tau^n\nu^\epsilon:\;n\in Z,\,\epsilon\in\{0,1\}\}.$$
Then we  receive our subgroups  as the ``modular part''(means intersection with $\mc M$) of the Neumann subgroups of  $\wh{\mc M}$ which are not contained in $\mc M$.  All our groups have two orbits of $\mc T$ in their coset graphs. However the constructions of Brenner and Lyndon and our give disjoint classes.

In Section~2 we consider anisotropy subgroups of the extended modular group as analogue of nonparabolic subgroups of the modular group. It is shown that every Neumann subgroup of the extended modular group is a maximal anisotropy subgroup. We also characterize Neumann subgroups of the extended modular group as those generated by special involution of $Z$ and using this involution we find the specific presentation which we use in the next section to exhibit the structure of investigated subgroups.

  In Section~3 the main tool in our analysis of the structure nonparabolic subgroups of $\wh{\mc M}$ is a coset graph of a subgroup introduced in \cite{Sc} and used in \cite{BL1}, \cite{BL2}. Because the complete set of representatives of a Neumann subgroup $\wh{\mc S} \subseteq \wh{\mc M}$ are $\tau^{n}$ and $(\tau \nu)^{n}$, we take their cosets as vertices of the coset graph of the group $\wh{\mc S}$. The set of directed edges is determined by actions of $\omega$, $\varphi$ and $\nu$ on vertices of a graph. In the paper these actions are denoted by A, B and  V respectively and comes from the definition of Neumann group as an abstract quadruple (\cite{BL1}) with suitable permutations such, that $C=AB$ acts semi-transitively (Theorem \ref{Nsym}).
  The coset graph of a Neumann subgroup of $\mc M$ is a directed graph and consists of the Eulerian path determined by the action of $C=AB$. The action of $C$ on the coset graph of a Neumann subgroup of $\wh{\mc M}$ has precisely two infinite orbits, i.e. $\{\wh{\mc M}\tau^n:\;n\in Z\}$, $\{\wh{\mc M}\tau^n\nu:\;n\in Z\}$, and in this sense the graph is ``quasi-Eulerian''. Another isomorphic coset graph  of a subgroup $\wh{\mc S} \subseteq \wh{\mc M}$ can be obtained using as edges A, $B^{-1}$ and V. The involution $AV$ determines an isomorphism of these two coset graphs preserving the structure of ''quasi-Eulerian'' paths. In the paper \cite{BL1} coset graph of a Neumann group is used to determine independent set of generators. It was possible because there is one to one labeling of edges by generators of a group such, that edges in any vertex determines a relation  between generators. In other words the coset graph of a subgroup can be interpreted as a presentation of this subgroup. We characterize Neumann subgroups of the extended modular group as those which have the special presentation with $\{\Sigma_n:\;n\in Z\}$ as the set of generators  (4. of Proposition~\ref{prez}). In the labeling of the coset graph  of subgroup $\wh{\mc S}$ each generator $\Sigma_{n}$ appears twice so we can not use him for finding independent generators. The appropriate graph $\Gamma (\wh{\mc S})$ we obtain as a homomorphic image of the coset graph  by the action of the isomorphism $AV$. In the Lemma \ref{isograph} we prove that graphs $\Gamma (\wh{\mc S})$ and the coset graph $\Gamma_{ \mc M}(\mc S)$ of the group $S=\wh{\mc S} \cap \mc M$ in $\mc M$ are isomorphic. Following \cite{BL1} we found an independent set of generators of $\wh{\mc S}$ from the graph $\wt{\Gamma}(\wh{\mc S})$ and of $\mc S$ from $\Gamma_{ \mc M}(\mc S)$ respectively.

\section{Neumann subgroups of the extended modular group}
In this section we introduce the notion of Neumann groups in $\wh{\mc M}$ and provide some equivalent characterizations. Recall that a Neumann subgroup of $\mc M$ is defined as a complement of $\mc T$. It is proved in \cite{BL1} that conjugates of $\mc T$ are precisely maximal parabolic groups and that a Neumann subgroup is a complement of every maximal parabolic one. Since parabolic elements are all in $\mc M$ we cannot literally transpose the definition  of a Neumann groups to $\wh{\mc M}$.
 Note that $\mc T=\mc M_{(\infty)}$ is an isotropy group of $\infty$, hence every maximal parabolic subgroup of $\mc M$ is equal to $\mc M_{(p)}$, an isotropy group of  some $p\in Q^*$. For our purpose the notion of isotropy subgroups in $\wh{\mc M}$ stands for a generalization of maximal parabolic subgroups in $\mc M$.
\begin{remark}
It is not difficult to observe that $\widehat{\mc S}<\wh{\mc M}$ is an isotropy subgroup of some $p\in Q^*$ iff $\widehat{\mc S}$ is a conjugation of the isotropy subgroup $\wh{\mc T}=\wh{\mc M}_{(\infty)}=\langle\tau, \nu\rangle$ of $\infty$.
\end{remark}
\begin{definition} A \emph{Neumann subgroup} of $\wh{\mc M}$ is a complement $\widehat{\mc S}$ of $\wh{\mc T}$, i.e. $\widehat{\mc S}\,\wh{\mc T}=\wh{\mc M}$ and $\widehat{\mc S}\cap\wh{\mc T}=1$.
\end{definition}
Since $\wh{\mc T}=\{\tau^n\nu^\epsilon:\;n\in Z,\,\epsilon\in\{0,1\}\}$ the property that $\tau^n, \tau^n\nu$, $n\in Z$,  form a complete system of distinct  right coset representatives of $\widehat{\mc S}$ in $\wh{\mc M}$ is an equivalent definition of Neumann subgroups of $\wh{\mc M}$.

In original Neumann's paper \cite{N} he investigates subgroups of $SL(2,Z)$ defined by two conditions:
\begin{enumerate}
\item the subgroup contains matrices with any ordered pair $(a,c)$ of relatively prime integers as the first column;
\item no two matrices of the subgroup have the same first column.
\end{enumerate}
Consider the same definition in $GL(2,Z)$. In order to be able to project this notion onto $\wh{\mc M}\simeq PGL(2,Z)$ we additionally assume that if a vector $v$ corresponds to a matrix $A$, then $-v$ corresponds to $-A$. Then our definition is equivalent to Neumann's subgroups. Indeed, the condition 1. is equivalent to $\widehat{\mc S}\,\wh{\mc T}=\wh{\mc M}$ and the condition 2. is equivalent to  $\widehat{\mc S}\cap\wh{\mc T}=1$.

In \cite{BL1}  there are proved some properties of maximal parabolic subgroups of $M$ that can be almost automatically transfer to isotropy subgroups of $\wh{\mc M}$:
\begin{lemma}\label{basixep}
\begin{enumerate}
\item The following are equivalent:
\begin{enumerate}
\item $\widehat{\mc S}$ acts transitively on $Q^*$.
\item $\widehat{\mc S}\,\wh{\mc P}=\wh{\mc M}$ for some isotropy subgroup $\wh{\mc P}$.
\item $\widehat{\mc S}\,\wh{\mc P}=\wh{\mc M}$ for all isotropy subgroup $\wh{\mc P}$.
\end{enumerate}
\item If $\widehat{\mc S}$ acts transitively on $Q^*$ and $\widehat{\mc S}\cap\wh{\mc P}=1$ for some isotropy subgroup $\wh{\mc P}$, then $\widehat{\mc S}\cap\wh{\mc P}=1$ for all isotropy subgroup $\wh{\mc P}$
\end{enumerate}
\end{lemma}
\begin{corollary}\label{np}
A Neumann subgroup of $\wh{\mc M}$ is a complement to every isotropy subgroup of $\wh{\mc M}$.
\end{corollary}
\begin{definition}
A subgroup of $\wh{\mc M}$ is called \emph{anisotropy subgroup} if it contains no elements of isotropy subgroups but the unity.
\end{definition}
As a consequence of Lemma~\ref{basixep} we immediately get the following.
\begin{proposition}\label{Nmnp}
Every Neumann subgroup of $\wh{\mc M}$ is a maximal anisotropy subgroup of $\wh{\mc M}$.
\end{proposition}
Inspired by \cite{BL1}, we present the characterization of Neumann subgroups of $\wh{\mc M}$ in terms of permutations of some abstract set. By a \emph{quadruple} $(\Omega, A,B,V)$ we mean one where $\Omega=\Omega_0\cup\Omega_1$ with $\Omega_0$, $\Omega_1$ infinite, disjoint sets and $A$, $B$, $V$ are permutations of $\Omega$ such that
$$A^2=B^3=V^2=(AV)^2=(ABV)^2=1,$$
and moreover $\Omega_0V=\Omega_1$. Here, following \cite{BL1}, by the product $GH$   in Sym$(\Omega)$ we admit $H\circ G$ and for $G\in\,\mbox{Sym}(\Omega)$ and $p\in\Omega$, $pG$ stands for $G(p)$.\\
A  quadruple $(\Omega, A,B,V)$ is  called \emph{semi-transitive} if $C=AB$ acts transitively on $\Omega_j$, $j=0,1$.

Using the presentation of $\wh{\mc M}=\langle\omega, \omega\tau, \nu\rangle$ we can see that the assignment
\begin{equation}\label{quadruple}
A=\omega\Phi,\; B=(\omega\tau)\Phi,\; V=\nu\Phi\in\mbox{Sym}(\Omega)
\end{equation}
gives an injection of $\wh{\mc M}$ into Sym$(\Omega)$ that we also denote by $\Phi$.
\begin{theorem}\label{Nsym}
The conjugacy classes of Neumann groups in $\wh{\mc M}$ (that are not contained in $\mc M$) are in one-to-one correspondence with the isomorphism classes of semi-transitive quadruples $(\Omega, A,B,V)$ with $\Omega_j$ infinite countable.
\end{theorem}
\begin{proof}
Let $\widehat{\mc S}<\wh{\mc M}$ be a Neumann group, hence a compliment to $\wh{\mc T}$. Let $\Omega_j=\{\widehat{\mc S}\tau^n\nu^j\}$, $j=0,1$, and $\Omega=\Omega_0\cup\Omega_1$. Then the action of $\wh{\mc M}$ on $\Omega$ by right multiplication defines a map
\begin{equation}\label{Phi}
\Phi:\wh{\mc M}\longrightarrow \mbox{Sym}(\Omega).
\end{equation}
Put $$A=\omega\Phi,\; B=(\omega\tau)\Phi,\; V=\nu\Phi.$$ Then we have
$$(\widehat{\mc S}\tau^n\nu^j)\,C=(\widehat{\mc S}\tau^n\nu^j)(\tau\Phi)
=\widehat{\mc S}\tau^{n+(-1)^j}\nu^j$$
and
$$(\widehat{\mc S}\tau^n\nu^j)\,V=(\widehat{\mc S}\tau^n\nu^j)(\nu\Phi)=\widehat{\mc S}\tau^n\nu^{j+1}.$$ We have found the corresponding quadruple  $(\Omega, A,B,V)$. \\
Since $\widehat{\mc S}\,\wh{\mc T}=\wh{\mc M}$, every conjugate of $\widehat{\mc S}$ is of the form $\tau^{-k}\nu^i\widehat{\mc S}\nu^i\tau^k$, $i=0,1$. The cosets of this conjugate are of the form  $\tau^{-k}\nu^i\widehat{\mc S}\tau^{l}\nu^j$, $i,j=0,1$ and make up the set $\Omega'$. Therefore the correspondence
$$\Omega'\ni\tau^{-k}\nu^i\widehat{\mc S}\tau^{l}\nu^{j}\longmapsto
 \widehat{\mc S}\tau^{l}\nu^j\in\Omega$$
 induces the isomorphism between appropriate quadruples.

 Now let $(\Omega, A,B,V)$ be a given quadruple. Consider the injection defined by \ref{quadruple}. Choose an element $p\in\Omega$ and define
 $$\widehat{\mc S}=\{\alpha\in\wh{\mc M}:\;\wh p(\alpha\Phi)=\wh p\}.$$
 We call $\wh p$ a \emph{stabilized element} of the quadruple of $\wh{\mc S}$.
 Then $\wh p(\tau^k\Phi)=\wh pC^k\neq\wh  p$, provided $k\neq0$, since $C$ is transitive on both of $\Omega_j$. Since $\Omega_0V=\Omega_1$ and $\Omega_jC=\Omega_j$, $\wh p\neq \wh p(\tau^k\nu\Phi)=(\wh pV)C^{-k}$, $k\in Z$. Therefore $\widehat{\mc S}\cap\wh{\mc T}=1$. \\
 Moreover if $\alpha\in\wh{\mc M}$ then, by the semi-transitivity property,
 $\wh p(\alpha\Phi)$ equals to $\wh pC^k=\wh p(\tau^k\Phi)$ or $\wh pC^kV=\wh p(\tau^k\nu\Phi)$, whence $\wh p=\wh p(\alpha\tau^{-k}\Phi)$ or $\wh p=\wh p(\alpha\tau^{-k}\nu\Phi)$ and we get that $\alpha\in\widehat{\mc S}\tau^k$ or $\alpha\in\widehat{\mc S}\tau^k\nu$. Therefore
 $\widehat{\mc S}\,\wh{\mc T}=\wh{\mc M}$ and $\widehat{\mc S}$ is a Neumann subgroup.  It is easy to see that different choices of $\wh p\in\Omega$ give  conjugate Neumann subgroups.
\end{proof}
From now on we will identify subgroups of $\wh{\mc M}$ with their images by $\Phi$.
\begin{remark}\label{inwolucja}
With a given Neumann subgroup of $\mc M$ there is associated some involution of $Z$, which was discovered by Magnus \cite{Ma} and Tretkoff \cite{T}. Independently, in more general case of anisotropy subgroups of $\wh{\mc M}$, this involution was explored in \cite{MS1}, \cite{MS2} and \cite{MS3}. Using this involution we can construct a set of generators for an anisotropy subgroup without using the Reidemeister-Schreier process. Because we will use those generators as well as an involution, we give  definitions here.
\end{remark}

Let $(\Omega, A,B,V)$ be a quadruple of $\wh{\mc S}$ and  $\wh p\in\Omega$ be a stabilized element. Let the maps $\iota:Z\longrightarrow Z$ and $\delta_.:Z\longrightarrow\{-1,1\}$ be determined by
$$\wh pC^nA=\wh pC^{\iota(n)}V^\frac{1-\delta_n}{2},\;n\in Z.$$
Define
$$\Sigma_n=C^nAV^\frac{1-\delta_n}{2}C^{-\iota(n)}\in\wh{\mc S},\;\;\;n\in Z.$$
\begin{proposition}\label{prez}
Let $\iota$ and $\Sigma_n$ be defined as above. Then
\begin{enumerate}
\item $\iota$ is an involution on $Z$ and $\delta_{\iota(n)}=\delta_n$;
\item if  $pC^kV^{\epsilon_k}$, $pC^lV^{\epsilon_l}$ and $pC^mV^{\epsilon_m}$  are three vertices in the same $B$-cycle then $\delta_k\delta_l\delta_m=1$ (equivalently, $\delta_{\pi(l)}\delta_{\pi(m)}-\delta_{\pi(k)}=0$ for every $\pi\in S_3$);
\item \begin{equation}\label{eqiota}\iota(\iota(n)+\delta_n)=\iota(n-1)-\delta_{n-1}\end{equation}
for every $n\in Z$;
\item \begin{equation}\label{prN}
\langle\Sigma_n\; \mid \;\Sigma_n\Sigma_{\iota(n)+\delta_n}\Sigma_{\iota(n-1)}=1,\;n\in Z\rangle
 \end{equation}
is a presentation of $\wh{\mc S}$.
 \end{enumerate}
\end{proposition}
\begin{proof}
1. We have
$$\wh pC^n=\wh pC^nA^2=\wh pC^{\iota(n)}V^{\frac{1-\delta_n}{2}}A=
\wh pC^{\iota(\iota(n))}V^{\frac{1-\delta_{\iota(n)}}{2}}V^{\frac{1-\delta_n}{2}},$$
hence
$$C^{\iota(\iota(n))}V^{\frac{1-\delta_{\iota(n)}}{2}}V^{\frac{1-\delta_n}{2}}C^{-n}\in\wh{\mc S}.$$
Since $\wh{\mc S}$ is an anisotropy subgroup, it follows the thesis.\\
2. Since
$$\wh pC^nV^{\epsilon}B=\wh pC^nAV^{\epsilon}C=\wh pC^{\iota(n)}V^{\frac{1-\delta_n}{2}}V^{\epsilon}C=
\wh pC^{\iota(n)+(1-2\epsilon)\delta_n}V^{\frac{1-\delta_{\iota(n)}}{2}+\epsilon},$$
the action of $B$ on $\wh pC^n$  changes the orbit iff $\delta_n=-1$. Therefore in the $B$-cycle the orbit is changed either zero or two times, which immediately implies the thesis.\\
3. Observe that $C^{-1}AC^{-1}=BA$. We have
$$\wh pC^n(C^{-1}AC^{-1})=\wh pC^{\iota(n-1)}V^{\frac{1-\delta_{n-1}}{2}}C^{-1}=\wh pC^{\iota(n-1)-\delta_{n-1}}V^{\frac{1-\delta_{n-1}}{2}}.$$
On the other hand
$$\wh pC^n(BA)=\wh pC^{\iota(n)+\delta_n}AV^{\frac{1-\delta_{n}}{2}}=\wh pC^{\iota(\iota(n)+\delta_n)}V^{\frac{1-\delta_{\iota(n)+\delta_n}}{2}}V^{\frac{1-\delta_{n}}{2}},$$
hence
$$C^{\iota(\iota(n)+\delta_n)}V^{\frac{1-\delta_{\iota(n)+\delta_n}}{2}}V^{\frac{1-\delta_{n}}{2}}
V^{\frac{1-\delta_{n-1}}{2}}C^{-(\iota(n-1)-\delta_{n-1})}\in\wh{\mc S}.$$
Since $\wh{\mc S}$ is an anisotropy subgroup,  the thesis follows.\\
4. Since $\{C^n, C^nV\}$ is a complete  system  of  distinct  right
coset  representatives $\wh{\mc S}$ in $\wh{\mc M}$, using the previous points and the Reidemeister-Schreier rewriting process, which is somewhat laborious, we  end the proof. We show two steps of the process, for example. First we show how the relation  $\Sigma_n\Sigma_{\iota(n)}=1$ arises from $A^2=1$:
$$\begin{array}{ll}
1&=C^nA^2C^{-n}=(C^nAV^{\frac{1-\delta_{n}}{2}}C^{-\iota(n)})C^{\iota(n)}AV^{\frac{1-\delta_{n}}{2}}C^{-n}\\&\\
 &= (C^nAV^{\frac{1-\delta_{n}}{2}}C^{-\iota(n)})\cdot(C^{\iota(n)}AV^{\frac{1-\delta_{\iota(n)}}{2}}C^{-n})\\&\\
 &=\Sigma_n\Sigma_{\iota(n)}.
 \end{array}$$
Then we show how $\Sigma_n\Sigma_{\iota(n)+\delta_n}\Sigma_{\iota(n-1)}=1$ arises from $B^3=1$:
$$\begin{array}{ll}
1&=C^nB^3C^{-n}=\Sigma_nC^{\iota(n)}V^{\frac{1-\delta_{\iota(n)}}{2}}CB^2C^{-n}=
\Sigma_nC^{\iota(n)+\delta_n}AV^{\frac{1-\delta_{\iota(n)}}{2}}CBC^{-n} \\&\\
 &=\Sigma_n\Sigma_{\iota(n)+\delta_n}C^{\iota(n-1)}C^{(\delta_n\cdot(\delta_{\iota(n)+\delta_n})-\delta_{n-1})}V^{\frac{1-\delta_{\iota(n)+\delta_n}}{2}}V^{\frac{1-\delta_{\iota(n)}}{2}}BC^{-n} \\&\\
 &= \Sigma_n\Sigma_{\iota(n)+\delta_n}C^{\iota(n-1)}AV^{\frac{1-\delta_{\iota(n)+\delta_n}}{2}}V^{\frac{1-\delta_{\iota(n)}}{2}}C^{-n+1}\\&\\
 &= \Sigma_n\Sigma_{\iota(n)+\delta_n}\Sigma_{\iota(n-1)}C^{n-1}
 V^{\frac{1-\delta_{\iota(n-1)}}{2}}V^{\frac{1-\delta_{\iota(n)+\delta_n}}{2}}
 V^{\frac{1-\delta_{\iota(n)}}{2}}C^{-n+1}\\&\\
 &=\Sigma_n\Sigma_{\iota(n)+\delta_n}\Sigma_{\iota(n-1)}.
 \end{array}$$
\end{proof}
\begin{remark}
In the next section we will show that in fact $\wh{\mc S}$ is a free product and that its independent  generators may be chosen from among ${\Sigma_n}'s$.
\end{remark}
\begin{remark}\label{examp}
We can easily extend the construction of \cite{BL2} to find maximal anisotropy groups that are not Neumann subgroups of $\wh{\mc M}$. Namely, instead of taking the crystallographic group $Q$ one may consider the full symmetry group of the regular tessellation of the Euclidean plane by hexagons. In that way we obtain the groups that contain the groups constructed in \cite{BL2} as normal subgroups of index 2.
\end{remark}
\section{Coset graphs and structure of Neumann subgroups of $\wh{\mc M}$}
The structure of Neumann subgroups of the modular group is completely described in \cite{BL1} or \cite{S}. It is shown there that every Neumann subgroup of $\mc M$ is necessarily a free product of $r_\infty$ infinite cyclic  groups, $r_2$ groups of order $2$ and $r_3$ groups of order $3$ subject to the conditions that  $r_2+r_3+r_\infty=\infty$ and if $r_\infty$ is finite then it is even. Moreover every scenario can be realized.
Now assume that our Neumann subgroup $\wh{\mc S}$ of $\wh{\mc M}$ is not contained in $\mc M$. Obviously $\mc S=\wh{\mc S}\cap \mc M$ has index two in $\wh{\mc S}$. Moreover, it turns out that $\mc S$ is a maximal nonparabolic and not Neumann subgroup but  the structure dramatically changes which is described in Theorem~\ref{NnN}.

Similarly like in \cite{BL1} and \cite{BL2} our main tool in this section is the Schreier coset graph. Let us recall the notion. Let $\mc G$ be a group and $L\subset \mc G$ be a set of its generators. Given a subgroup $\mc H<\mc G$ we define its \emph{coset graph} $\Gamma=\Gamma_{\mc G}(\mc H)=\Gamma(\mc H)$ putting $V(\Gamma)=\{\mc Hg:\;g\in\mc G\}$ and allow a pair   $(\mc Hg,\mc Hh)$ of vertices to be connected by the $l$-edge, if $\mc Hgl=\mc Hh$. In general, a graph defined in this way is a multigraph. However, if $l$ is an involution then we identify $l$-edge with its inverse and treat it as a undirected edge. We also allow loops in a coset graph corresponding to fix points of the action of generators.\\
In the previous section we described  Neumann subgroups of $\wh{\mc M}$ in terms of abstract structures called quadruples. In fact every quadruple determines a coset graph if  we consider $\Omega$ as a set of vertices  and then  the actions of $A$, $B$, $V$ on $\Omega$ give appropriate labeled edges in it. In general the opposite is not true. The same graph can admit distinct quadruples.\\
In fact every coset graph indicates a subgroup up to a conjugation. In order to express   this fact precisely we follow the idea of Brenner and Lyndon \cite{BL2} and introduce the notion of the abstract $(2,3,2)$-graph. A connected graph $\Gamma$ is a $(2,3,2)$\emph{-graph}  if its set of vertices is countable and its set of edges is divided into three disjoint sets: undirected $A$-edges, undirected $V$-edges and directed $B$-edges, subject to the following conditions:
\begin{enumerate}
\item at each vertex there is exactly one $A$-edge and one $V$-edge;
\item each vertex is a source of the precisely  one $B$-edge;
\item at each vertex , the $B$-edge is either a loop, or is one in a cycle of three $B$-edges;
\item at each vertex, every  AVAV-path and every   ABVABV-path  is a cycle.
\end{enumerate}
Observe that there is the natural action of $PGL(2,Z)$ on a $(2,3,2)$-graph given by the actions $A$, $B$ and $V$.

The proof of the following is analogous to the proof of  Theorem~\ref{Nsym} and we omit it.
\begin{proposition} There is a bijective correspondence  between the isomorphism
classes of $(2,3,2)$-graphs and the conjugacy classes of subgroups of $\wh{\mc M}$.
\end{proposition}

For our purpose it is convenient to consider the coset graph of a given subgroup of $\wh{\mc M}$  without $V$-edges.   After \cite{BL2} we  call such  graphs $(2,3)$\emph{-graphs}. From  \cite{BL2} it follows that this graph should be a coset graph of some subgroup of  $\mc M$. We will see soon that this subgroup is equal to $\wh{\mc S}\cap\mc M$. For a subgroup $\wh{\mc S}$ of $\wh{\mc M}$ such a graph will be denoted by $\Gamma(\wh{\mc S})$.\\
If the actions of $A$ and $B$ have no fixed points then $\Gamma(\wh{\mc S})$ is a cubic graph. If $v=vA$, then we emphasize   that still there  is a precisely one $A$-edge at $v$.  If $v=vB$ than $v$ is a vertex of degree 1 and in such a case $\Gamma(\wh{\mc S})$ is cuboidal.

We have the following analogue of Corollary~1.7 from \cite{BL2}.
\begin{theorem}\label{Ncg}
A subgroup $\wh{\mc S}$ of $\wh{\mc M}$ is a Neumann subgroup  if and only if $\Gamma(\wh{\mc S})$ contains precisely two $C$-orbit which both are infinite and exchanged by the action of $V$.
Moreover a Neumann subgroup $\wh{\mc S}$ is not entirely contained in $\mc M$ if and only if $\Gamma(\wh{\mc S})$ is connected.
\end{theorem}
\begin{proof}
 If $\wh{\mc S}$  is a Neumann subgroup then its coset graph, by definition, contains two infinite $C$-orbits which are exchanged by the action of $V$. \\
 Assume now that $\Gamma(\wh{\mc S})$ has two infinite $C$-orbits exchanged by the action of $V$ , then one of them is $\{\wh{\mc S}C^n:\;\in Z\}$ since there is no finite $C$-orbits. If the orbits     $\{\wh{\mc S}C^n:\;\in Z\}$ and $\{\wh{\mc S}VC^n:\;\in Z\}$ were  the same then it would be preserved by the action of $V$, contrary to the assumption. Therefore $\wh{\mc S}$ is a Neumann subgroup.

If $\wh{\mc S}<\mc M$ is a Neumann subgroup then we can consider two coset graphs of $\wh{\mc S}$, one with respect to $\mc M$, denoted by $\Gamma_{\mc M}(\wh{\mc S})$, and the other one with respect to $\wh{\mc M}$. Considering the second one we have add to $V(\Gamma_{\mc M}(\wh{\mc S}))$ the cosets arising by right action by $V$ to get $\Omega=V(\Gamma(\wh{\mc S}))$. Obviously $V(\Gamma_{\mc M}(\wh{\mc S}))$ and the added set of vertices are two disjoint (recall that we have abandoned $V$-edges) connected components of $\Gamma(\wh{\mc S})$.

Let  $\wh{\mc S}$ be a Neumann subgroup not contained in $\mc M$. Then, by definition, $\Omega$ consists of two $C$-orbits, namely $\Omega_j=\{pC^nV^j:\;n\in Z\}$, $j=0,1$, which form two sets of connected components. By Remark~\ref{inwolucja} there is $n\in Z$ such that $\delta_n=-1$. Therefore we have
$$pC^nA=p\Sigma_nC^{\iota(n)}V=pC^{\iota(n)}V.$$
It follows that the two $C$-orbits are connected by the $A$-edge, hence $\Gamma(\wh{\mc S})$ is connected.
\end{proof}
\begin{example}
We cannot abandone the assumption about the action of $V$ in the above theorem since there are subgroups of the extended modular group with the coset graph consisting of two infinite $C$-orbits being  not Neumann subgroups. Such a subgroup has to be necessarily isotropic. The one with the simplest coset graph (depicted in the Fig.~1) is the free product of the groups of order two with the set of independent generators equal to
$$\{V,\, C^{3k}AC^{-3k},\,C^{3k+2}BAB^{-1}C^{-3k-2}:\;k\in Z\}.$$
The two orbits are equal to $\{\wh{\mc S}C^n:\;n\in Z\}$ and    $\{\wh{\mc S}B^{-1}C^n:\;n\in Z\}$.
\end{example}
 \begin{figure}[ht]{\footnotesize \textbf{ Fig. 1}\hspace{2mm} The coset graph of the non-Neumann subgroup of $\wh{\mc S}$ consisting of two $C$-orbits (in fact a pair of quasi-Eulerian paths). The coset $\wh{\mc S}$ is situated at the enlarged vertex; $A$-edges are depicted as bold lines; directed $B$-edges  as thin lines and $V$-edges as dashed lines.  }
 \centering
     \includegraphics[width=.9\textwidth]{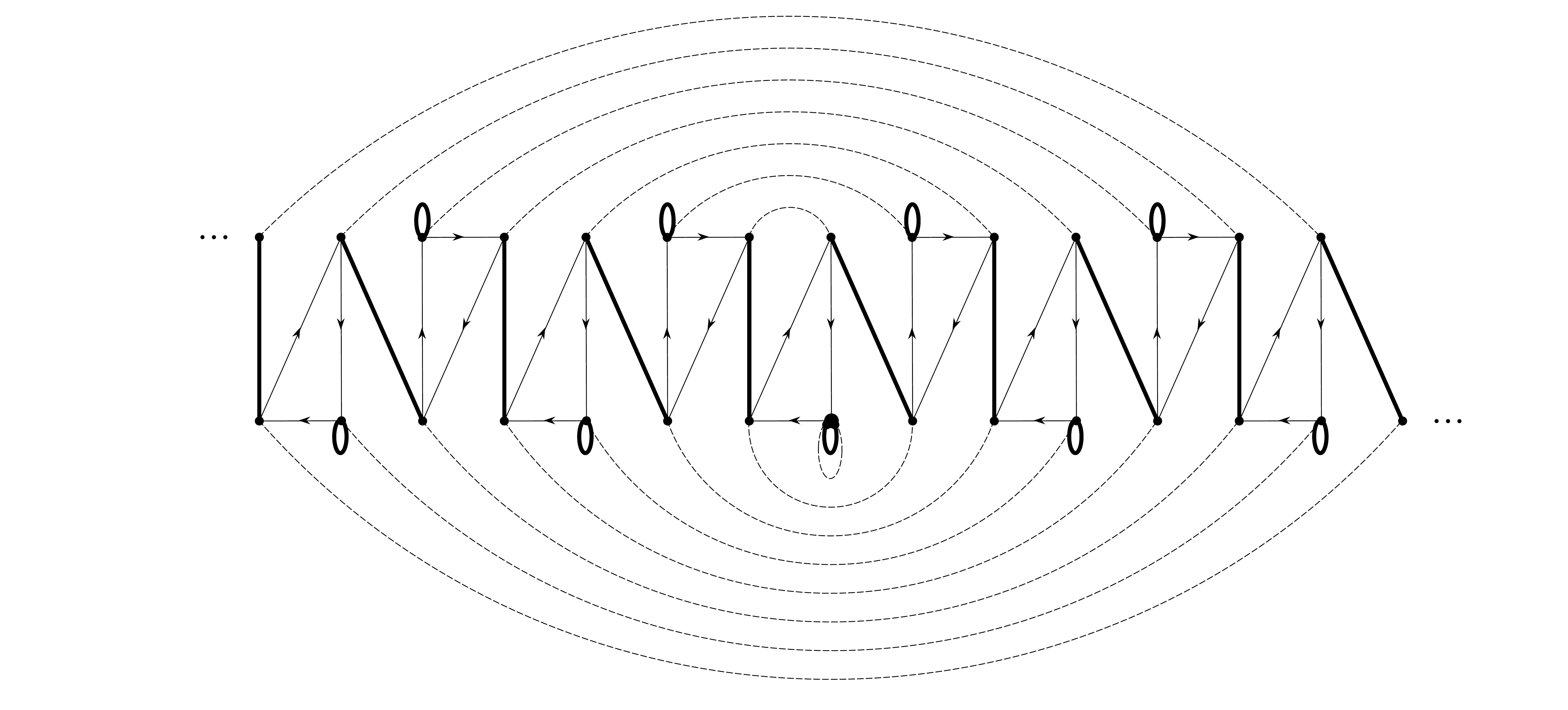}
\end{figure}

\begin{remark}\label{N12o}
 If $\wh{\mc S}<\mc M$ is a Neumann subgroup of $\wh{\mc M}$ then two disjoint connected components of $\Gamma(\wh{\mc S})$ are isomorphic in the following sense. The  action of $AV$ on $\Omega_0$ is a bijections which carries this set onto $\Omega_1$, $A$-edges  in $\Omega_0$ onto $A$-edges in $\Omega_1$ and $B$-edges in $\Omega_0$ onto $B^{-1}$-edges in $\Omega_1$.\\
 When we consider the coset graph of $\wh{\mc S}$ in $\mc M$ we have to identify those two orbits. The fact that  a subgroup of $\mc M$ is a Neumann one iff its coset graph in $\mc M$ consists of precisely one $C$-orbit is proved in \cite{BL1}.
 \end{remark}

 Given a subgroup $\wh{\mc S}$ of $\wh{\mc M}$ we consider a subgroup
$$\mc S=\wh{\mc S}\cap\mc  M$$
of $\mc M$.
From now on we always assume that $\wh{\mc S}$ is not contained in $\mc M$.
\begin{remark}
We do not know yet whether a subgroup of $\mc M$ can be an intersection of two distinct Neumann subgroups of $\wh{\mc M}$. However, we are able to show, using the methods similar to those used in the proof of Theorem~\ref{maxnonpar}, that if it is a case then there is a positive integer  $k$ such that $(AB)^k(BA)^{-k}\in \mc S$. Moreover, it can be showed that if $k\in\{1,2,3,4,5\}$ then $C^{mk}\in \mc S$ for some positive integer $m$, which contradicts definition of $\mc S$. Also we are able to show that $k$ has to be even. Finally, we know that it is imposible if there is finitely many $\Sigma_n$'s of negative determinant for both $\wh{\mc S}$'s.
\end{remark}
Properly understood coset graphs of  $\wh{\mc S}$ and $\mc S$ are isomorphic.
\begin{lemma}\label{isograph}
The cosets graphs $\Gamma(\wh{\mc S})$  and $\Gamma_{\mc M}(\mc S)$ of  $\mc S$ in $ \mc M$ are isomorphic.
\end{lemma}
\begin{proof}
Given a coset $\wh{\mc S}D$, it is a vertex of   $\Gamma(\wh{\mc S})$, we assume that $D\in\mc M$. Then the isomorphism between
 $\Gamma(\wh{\mc S})$ and $\Gamma_{\mc M}(\mc S)$ is given by
 $$\wh{\mc S}D\longmapsto \mc SD.$$
\end{proof}
 The following proposition yields the basic description of the structure of Neumann subgroups of the extended modular group.
\begin{proposition}\label{2}
 If $\wh{\mc S}$ is a Neumann subgroup of $\wh{\mc M}$    then
 \begin{enumerate}
 \item  $\mc S$ is a normal subgroup in $\wh{\mc S}$ of index two;
 \item  $\wh{\mc S}$ is never a semi-direct product of $\mc S$ and a subgroup isomorphic to $\mc C_2$;
 \item  $\wh{\mc S}$ is an infinite free product of groups isomorphic to $\mc C_2$ or $\mc C_3$ or the infinite cyclic group.
 \end{enumerate}
 \end{proposition}
\begin{proof}
1. Obviously $\mc S\lhd\wh{\mc S}$ and $\wh{\mc S}/\mc S\simeq\mc  C_2$.\\
2. If $\wh{\mc S}$ was  a semi-direct product of $\mc S$ and a subgroup isomorphic to $\mc C_2$ then $\wh{\mc S}$ would contain an involution of negative determinant which is conjugate to either $CV$ or $V$, hence is a member of an isotropy subgroup of $\wh{\mc M}$, a contradiction, by Corollary~\ref{np}. \\
3. It is known that $\wh{\mc M} =D_2\star_{\mathbb Z_2}D_3$ is a free product of $D_2=C_2\times C_2$ and $D_3=C_3\rtimes C_2$ amalgamated over the common subgroup generated by the involution of negative determinant. Since $\wh{\mc S}$ does not contain such an involution,    the thesis follows from the generalization of the Kurosh theorem (see i.e. \cite{KS}).
 \end{proof}
\begin{remark}
As a byproduct of the proof of Theorem~\ref{NnN}, using technics involving coset graphs, we will obtain an independent on the generalized Kurosh theorem proof of 3 of the above proposition.
\end{remark}

In \cite{BL2} there is constructed  a wide class of maximal nonparabolic subgroups  that are not Neumann subgroups of $\mc M$. These groups have the cosets graphs consisting of even or infinite number of (infinite) $C$-orbits. The following, one of the main result of the present paper, shows that Neumann subgroups of $\wh{\mc M}$ lead to completely disjoint class of such subgroups. Moreover in the last section we show that each possible structure in $\mc M$  as a free product is realized by the continuum  of non-conjugate   groups.
\begin{theorem}\label{maxnonpar}
If $\wh{\mc S}$ is a Neumann subgroup of $\wh{\mc M}$ then $\mc S$ is a maximal nonparabolic subgroup  that is not a Neumann subgroup of $\mc M$.
\end{theorem}
\begin{proof}
Assume that $\wh{\mc S}$ is  a Neumann subgroup of $\wh{\mc M}$. By Theorem~\ref{Ncg}, $\Gamma(\wh{\mc S})$ has precisely two infinite $C$-orbits. Therefore from Lemma~\ref{isograph} and Remark~\ref{N12o} it follows that $\mc S$ is not a Neumann subgroup of $\mc M$.

Let $(\wh\Omega, A,B,V)$ and $(\Omega,A, B)$ be a quadruple and  a triple associated with $\wh{\mc S}$ and $\mc S$, respectively. Recall that
$$\wh\Omega=\left\{\wh p\,C^kV^{\epsilon_k}:\;k\in Z,\epsilon_k\in\{0,1\}\right\}.$$
where $\wh p\,$ is the stabilized element of $(\wh\Omega, A,B,V)$.
If $\psi:\wh\Omega\longrightarrow\Omega$ is given by
$$\psi\left(\wh p\,C^kV^{\epsilon_k}\right)=pE^{\epsilon_k}C^kV^{\epsilon_k},$$
where $p$ is the stabilized element of $(\Omega,A, B)$ and $\wh S=S\cup SE$, then $\psi$ is an isomorphism between $\Gamma(\wh{\mc S})$ and $\Gamma_{\mc M}(\mc S)$ from  Lemma~\ref{isograph}.

Now take $D\in\mc M\setminus\mc S$ and denote $\mc U=\langle\mc S,D\rangle$. It is well-known (see i.e. \cite{BL2}), that $\Gamma_{\mc M}(\mc U)$ is a homomorphic image of $\Gamma_{\mc M}(\mc S)$. Let us denote the appropriate homomorphism by $\Lambda$. Then $q=\Lambda(p)$ is the stabilized element of the triple of $\mc U$. Note that all fibers of $\Lambda$ have the same cardinality and that $\{p,pD\}\subset\Lambda^{-1}(q)$. \\
We have to consider three cases:
\begin{enumerate}
\item $pD=pC^n$;
\item $pD=pEC^nV$ and $\{p,pD\}\subsetneq\Lambda^{-1}(q)$ ;
\item $pD=pEC^nV\;\;\;\mbox{ and }\;\;\;\{p,pD\}=\Lambda^{-1}(q).$
\end{enumerate}
If 1. holds then $C^n\in\mc U$, hence $\mc U$ contains   a parabolic element.\\ If 2. holds  then $\Lambda^{-1}(q)$ contains $pE^\epsilon C^{n'}V^\epsilon$ except $p$ and $pEC^nV$. It follows that in such a case $\mc U$ contains a parabolic element as well.\\ If 3. holds then $\Gamma_{\mc M}(\mc U)$ has a precisely one infinite $C$-orbit, hence, by Remark~\ref{N12o}, is a Neumann subgroup of $\mc M$. Let $\iota$ and $\wh\iota$ denote the associated with $\mc U$ and $\wh{\mc S}$, respectively, involutions of $Z$. Since $qC^kA=qC^{\iota(k)}$ and $\Lambda$ is a graph homomorphism, we have as well $\Lambda^{-1}(q)C^kA=\Lambda^{-1}(q)C^{\iota(k)}$. Now
$$\begin{array}{ll}
\Lambda^{-1}(q)C^kA&=\left\{pC^kA,pEC^{n-k}AV\right\}=\left\{\psi\left(\wh p\,C^kA\right),\psi\left(\wh p\,C^{n-k}AV\right)\right\}\\&\\
     &=\left\{\psi\left(\wh p\,C^{\wh\iota(k)}V^{\frac{1-\delta_k}{2}}\right),\psi\left(\wh p\,C^{\wh\iota(n-k)}V^{\frac{3-\delta_{n-k}}{2}}\right)\right\}\\&\\
     &=\left\{pE^{\frac{1-\delta_k}{2}}C^{\wh\iota(k)}V^{\frac{1-\delta_k}{2}},pE^{\frac{1+\delta_{n-k}}{2}}C^{\wh\iota(n-k)}V^{\frac{3-\delta_{n-k}}{2}}\right\}\\&\\
     \mbox{and} &\\&\\
\Lambda^{-1}(q)C^{\iota(k)}&=\left\{pC^{\iota(k)},pEC^{n-\iota(k)}V\right\}.
\end{array}$$
If $\delta_k=-1$ then we have
$$\left\{pC^{\iota(k)},pEC^{n-\iota(k)}V\right\}=\left\{pEC^{\wh\iota(k)}V,pE^{\frac{1+\delta_{n-k}}{2}}C^{\wh\iota(n-k)}V^{\frac{3-\delta_{n-k}}{2}}\right\}.$$
Since $E\in\wh{\mc S}$ and $\wh S$ intersects trivially with $\wh{\mc T}$, $\delta_{n-k}=-1$, hence
$$\left\{pC^{\iota(k)},pEC^{n-\iota(k)}V\right\}=\left\{pEC^{\wh\iota(k)}V,pC^{\wh\iota(n-k)}\right\}.$$
But $\mc S$ does not contain parabolic elements, thus we get
$$\iota(k)=\wh \iota(n-k)=n-\wh\iota(k).$$
If $\delta_k=1$ then similar arguments  show that $\delta_{n-k}=1$ and
$$\iota(k)=\wh\iota(k)=n-\wh\iota(n-k).$$
Substituting $k:=n-k$  in the two last equalities we get
\begin{equation}\label{iiii}
 \iota(k)+\iota(n-k)=n
 \end{equation}
for every $k\in Z$.\\
If $n=2m$ then from the above it follows that
$$\iota(m)=m\;\;\;\mbox{ and }\;\;\;\iota(m-1)+\iota(m+1)=2m.$$
Therefore, in virtue of  (\ref{eqiota}), we have
$$\iota(m+1)=\iota(\iota(m)+1)=\iota(m-1)-1,$$
which forces
$$2\iota(m+1)=2m-1,$$
a contradiction.\\
If $n=2m+1$  put $l=\iota(m+1)+1$. Then  (\ref{eqiota}) and (\ref{iiii}) yield
$$\begin{array}{ll}
(l+1)+(\iota(l)-1)&=l+\iota(l)=\iota(m+1)+1+\iota(\iota(m+1)+1)\\&\\
   &=\iota(m+1)+1+\iota(m)-1=2m+1.\end{array}$$
Applying (\ref{iiii}) one get
$$\iota(l+1)+\iota(\iota(l)-1)=2m+1.$$
Since $\iota(\iota(l)-1)=\iota(l+1)+1$,
$$m=\iota(l+1)=\iota(\iota(m+1)+2),$$
hence
$$\iota(m)-\iota(m+1)=2.$$
This equality, combined with $\iota(m)+\iota(m+1)=2m+1$, gives
$$2\iota(m)=2m+3,$$
a contradiction.

We showed that $\mc S$ is a maximal nonparabolic subgroup of $\mc M$.
\end{proof}
We aim for a description of structure of a given Neumann subgroup of $\wh{\mc M}$.\\
Let us denote $\Gamma=\Gamma(\wh{\mc S})\simeq\Gamma_{\mc M}(\mc S)$. Considering $\Gamma=\Gamma(\wh{\mc S})$ one can label the $A$-edge at the vertex $\wh p\,C^nV^\epsilon$ by $\{\Sigma_n, \Sigma_{\iota(n)}\}$.  The vertices $\wh p\,C^nV^\epsilon$ and $\wh p\,C^{\iota(n)}V^{\frac{1-\delta_n}{2}}V^\epsilon$ have the same $A$-edge which we consider as  a graphical illustration of the relation $\Sigma_n\Sigma_{\iota(n)}=1$. We adopt the convention that  the vertex $\wh p\,C^nV^\epsilon$ chooses from the label of its $A$-edge the generator $\Sigma_n$. Now if we start the $B$-cycle from $\wh p\,C^nV^\epsilon$ then, in virtue of 2. and 3.  of Proposition~\ref{prez}, the chosen by the consecutive  vertices generators  satisfy $$\Sigma_n\Sigma_{\iota(n)+(1-2\epsilon)\delta_n}\Sigma_{\iota(n-1+2\epsilon)}=1,$$
which is the relation given in 4. of Proposition~\ref{prez} in the appropriate form.  Therefore if $\ol\Gamma$ denotes the edge contraction of $\Gamma$ by the action of $B$ then $\ol\Gamma$ with the induced labeling can be understood as a graphical illustration of the presentation (\ref{prN}) of $\wh S$. Since the above labeling is two-to-one, we cannot use  $\ol\Gamma$ to indicate  independent generators of $\wh{\mc S}
$. In order to remove this obstacle observe that $AV$ carries each $B$-cycle in $\Gamma$ onto a $B$-cycle (reversing $B$-orientation) and each $A$-edge onto $A$-edge. It follows that $AV$ is an automorphism of $\ol\Gamma$. Let $\wt\Gamma$ denote the vertex contraction of $\ol\Gamma$ by the action of $AV$. Because the images of the edges with the same label always connect the same pair of vertices in $\wt\Gamma$, we identify them.\\
Although the action of $B$ on $\Gamma$ projects onto the  identity on $\ol\Gamma$ we still can consider the action of $C$ on $\ol\Gamma$ as a projection of the action of $C$ on $\Gamma$.\\
A pair of paths in the cuboid graph is \emph{quasi-Eulerian} if
\begin{itemize}
\item each of paths is infinite and reduced except at vertices of degree $1$;
\item every oriented edge connecting distinct vertices belongs to the precisely one of paths in the pair.
\end{itemize}
 If  $\Gamma=\Gamma(\wh{\mc S})$ consists of two orbits, which is the case when $\wh{\mc S}$ is assumed to be Neumann,  then $\ol \Gamma$ necessarily possess a  quasi-Eulerian  pair of paths.
\begin{lemma}\label{infst} Let $\wh{\mc S}$ be a Neumann subgroup of $\wh{\mc M}$.
If $\overline T$ is
\begin{enumerate}
\item either an attached cuboidal finite graph or
\item an attached infinite  simple tree
\end{enumerate}
 in $\overline{\Gamma}$ then the set of vertices of $\ol T$  is contained in the precisely one $C$-orbit.
\end{lemma}
\begin{proof} Since $\ol \Gamma$ is quasi-Eulerian, 1. is obvious.

Let $P:\Gamma\longrightarrow \ol\Gamma$ denote a natural projection. If $\ol T$ is infinite and attached at the vertex $v_0$ then  let $\wh p\, C^{n}V^{\epsilon}\in P^{-1}(v_0)$, $\epsilon\in\{0,1\}$, be such that $\wh p\, C^{n}V^{\epsilon}A\in P^{-1}\left(\ol T\right)$ and denote $T=P^{-1}\left(\ol T\setminus\{v_0\}\right)\cup\{\wh p\, C^{n}V^{\epsilon}\}$.  Since $\ol T$ is a simple infinite tree it consists of   a one-sided sequence $(v_n)_{n\geq 0}$ of vertices connected sequently by $P$-images of $A$-edges and finite trees attached to $v_n$'s. If $\ol T$, hence $T$ as well, intersects the both orbits then from 1. and the quasi-Eulerian property it follows  that for every $i\geq0$ there are $k_i,l_i\in Z$ such that $\wh p\, C^{k_i}V^{\epsilon}\in P^{-1}(v_i)$, $\wh p\,C^{l_i}V^{\epsilon+1}\in P^{-1}(v_{i+1})$ and $\wh p\, C^{k_i}V^{\epsilon}A=\wh p\,C^{l_i}V^{\epsilon+1}$. In particular we have $k_0=n$. Observe that
\begin{equation}\label{cn}
\wh p\,C^kV^\epsilon,\;\wh p\,C^lV^{\epsilon+1}\in T\mbox{ iff }\left[(-1)^\epsilon(k-n)\geq0\mbox{ and }(-1)^\epsilon(l-l_0)\geq0\right].
\end{equation}
Now since $\iota(n)=l_0$, we have that
$$\wh p\,C^{l_0}V^\epsilon A=\wh p\,C^{n}V^{\epsilon+1},$$
hence either both $\wh p\,C^{l_0}V^\epsilon$ and $\wh p\,C^{n}V^{\epsilon+1}$ are in $T$ or both are outside of $T$. Thus  (\ref{cn}) yields $n=l_0$ which, according to Corollary~\ref{np}, is a  contradiction.
\end{proof}
If an edge of $\Gamma$ is labeled by an involution then its $P$-image is a loop in $\ol\Gamma$, and every loop in $\ol\Gamma$ is of this form. The only vertices of valence $1$ are those whose $A$-edges are labeled by the pair of elements of order $3$.  Now remove from $\ol\Gamma$ all such loops and all vertices of valence $1$ and denote by $\ol\Gamma_0$ resulting graph. Let
$$
\wh L_0=\{\Sigma_n:\; \Sigma_n^2=1\mbox{ or }\Sigma_n^3=1\}.
$$
We already know that  $\wh L_0\subset\mc S$. \\
Let $\wt\Gamma_0$ be an image of $\ol\Gamma_0$ by the natural projection, i.e. $\wt\Gamma_0$ is obtained from $\wt\Gamma$ by removing all edges labeled by the elements of  $\wh L_0$.\\
From the above lemma one immediately get the following.
\begin{corollary}\label{betti}
If $\wh{\mc S}$ be a Neumann subgroup of $\wh{\mc M}$ then $\ol\Gamma_0$ and $\wt\Gamma_0$   have the  Betti number not less than $1$.
\end{corollary}
It is not difficult to observe that every cycle in $\ol\Gamma$ is transformed by $AV$ either to another cycle or to itself. Therefore from Lemma~\ref{infst} we immediately get the following.
\begin{corollary}\label{postac}
If $\wh{\mc S}$ is a Neumann subgroup of $\wh{\mc M}$ then $\ol\Gamma_0$ has the infinite Betti number  iff $\ol\Gamma_0$ is a cuboidal graph with no attached infinite trees. The same statement holds for $\wt\Gamma_0$.

 The Betti number of $\ol\Gamma_0$ is finite iff the Betti number of $\wt\Gamma_0$ is finite. In such a case   $\ol\Gamma_0$ is a finite cuboidal graph with precisely two infinite trees attached and  $\wt\Gamma_0$ is a finite cuboidal graph with precisely one infinite tree attached.
\end{corollary}

Now we are in a position to describe the structure of a Neumann subgroup of the extended modular group in details. We say that a subgroup of the extended modular has \emph{$(k,l,m)$-structure}, $k,l,m\in N_0\cup\{\infty\}$, if it is a free product of $k$ subgroups of order $2$, $l$ subgroups of order $3$ and $m$ infinite cyclic subgroups.
\begin{theorem}\label{NnN}
If  $\wh{\mc S}$ is a Neumann subgroup of $\wh{\mc M}$ that is not entirely contained in $\mc M$ then
  \begin{enumerate}
  \item $\wh{\mc S}$ has $(\wh r_2,\wh r_3,\wh r_\infty)$-structure subject to the conditions that
      \begin{enumerate}
      \item $\wh r_2+\wh r_3+\wh r_\infty=\infty$;
      \item  $\wh r_\infty\geq1$;
      \item if $\tilde \beta$ stands for the Betti number of $\wt\Gamma_0$, then $\wh r_\infty=\tilde\beta$;
       \item all independent generators of finite order are elliptic elements of the modular group.
            \end{enumerate}
    \item if $\mc S$ has $(r_2, r_3,r_\infty)$-structure  then
     \begin{enumerate}
      \item $(r_2,r_3,r_\infty)=(2\wh r_2,2\wh r_3,2\wh r_\infty-1)$;
      \item if $\beta$ stands for the Betti number of $\ol\Gamma_0$, then $r_\infty=\beta$.
      \end{enumerate}
      \end{enumerate}
\end{theorem}
\begin{proof}
 Recall that the labeling  of edges of $\wt\Gamma$ by elements of $\{\Sigma_n:\;n\in Z\}$ is bijective and  that at each vertex the cycle of edges gives a relation described in 4. of Theorem~\ref{prez}. Therefore we can use the method developed in \cite{BL1} to select the set of edges of $\wt\Gamma$ whose labels form a set of independent generators of $\wh{\mc S}$.

  Since the edges labeled by $\wh L_0$ can only occur at the ends of finite attached trees, the generators from $\wh L_0$  are independent. Of course all remaining edges are labeled by hyperbolic elements. \\
Note that we have removed from $\wt\Gamma$  all  loops labeled by involutions of the positive determinant but  still $\wt\Gamma_0$ may contain loops labeled by hyperbolic elements. We take those loops into account (as cycles) while we count the Betti number. \\
Let $\wt\Theta_1$ be a maximal tree in $\wt\Gamma_0$ and $\wh L_1$ be the set of labels of $E(\wt\Gamma_0)\setminus E(\wt\Theta_1)$. Let $\wt\Theta_2$ denote the union of minimal number, say $s\in N\cup\{\infty\}$, of simply infinite trees contained in $\wt\Theta_1$ and $\wh L_2$ be the set of labels of $E(\wt\Theta_1)\setminus E(\wt\Theta_2)$. Observe that the cardinality of $\wh L_1$ is equal to $\tilde\beta$ and the cardinality of $\wh L_2$ is equal to $s-1$ and that $\wh L_\infty=\wh L_1\cup \wh L_2$ consists of hyperbolic elements.

The elements of
$$\wh L=\wh L_0\cup \wh L_\infty$$
are independent and generate $\wh{\mc S}$.

We have $$\wh r_2+\wh r_3=\mid\wh L_0\mid\;\;\;\mbox{ and }\;\;\; \wh r_\infty=\mid\wh L_\infty\mid=\tilde\beta+s-1.$$  In virtue of Corollary~\ref{postac}, if $\tilde\beta<\infty$ then $s=1$, hence the latter equality above gives (c) of 1. If $\tilde\beta=\infty$ then there are infinitely many finite trees attached and it follows that $\mid\wh L_0\mid=\infty$, thus (a) of 1. holds.  Corollary~\ref{betti} implies (b) of 1. Now (d) of 1. follows from the definition of $\wh L_0$.

We pass to the proof of 2. Observe that if $E\in\wh{\mc S}\setminus\mc S$ then the set
$$\left\{ E^\epsilon\Sigma_nE^{-\epsilon\delta_n+\frac{\delta_n-1}{2}}:\;n\in\mathbb Z,\,\epsilon\in\{0,1\}\right\}$$
generates $\mc S$.

In order to realize the structure of  $\mc S$ take $E\in \wh L_\infty$ and consider the sets
$$L_0=\wh L_0\stackrel{\cdot}{\cup}E\wh L_0E^{-1},\;\;\;L_\infty=(E\wh L_\infty \stackrel{\cdot}{\cup}\wh L_\infty E^{-1})\setminus\{1\}.$$
Obviously elements of $L_0\cup L_\infty$ are independent generators of $\mc S$, therefore (a) and (b) of 2. follow immediately.
\end{proof}

In \cite{BL1} and \cite{S}, independently, it was shown that  any admissible structure for a Neumann subgroup of the modular group can be realized.  Analogously to the situation  it is possible to realize any admissible structure for a Neumann subgroup of the extended modular group. Recall that in order to describe some Neumann subgroup it is enough to define appropriate involution $\iota$ of the set of integers together with the sequence $\delta_n\in\{-1,1\}^\mathbb Z$ such that (\ref{eqiota}) is satisfied. The recursive method of such a construction is given in \cite{MS2} and \cite{MS3} .  For the sickness of completeness we describe this method below.\\
 A map $\tilde\iota:\{k,\ldots,k+l\}\longrightarrow\{k,\ldots,k+l\}$, $k\in \mathbb Z$, $l\geq 0$, is called a \emph{generating involution} if:
\begin{itemize}
  \item $\tilde\iota$ is an involution of $\{k,\ldots,k+l\}$;
  \item $\tilde\iota(k)=k+l$;
  \item $\tilde\iota$ satisfies (\ref{eqiota}) for each $n=k,\ldots,k+l-1$.
\end{itemize}
 Note that we can freely shift the domain of a generating involution along  integers.
 Given two generating involutions one may construct another one. Indeed, let $$\iota_j:\{k_j,\ldots,k_j+l_j\}\longrightarrow\{k_j,\ldots,k_j+l_j\},$$ $j=0,1$ be generating involutions with $k_1=k_0+l_0+1$.
 Define $$\tilde\iota=\iota_0\sqcup\iota_1:\{k_0-1,\ldots,k_1+l_1+1\}\longrightarrow\{k_0-1,\ldots,k_1+l_1+1\}$$ by
\begin{itemize}
  \item $\tilde\iota(k_0-1)=k_1+l_1+1$;
    \item $\tilde\iota\mid_{\{k_j,\ldots,k_j+l_j\}}=\iota_j$, $j=0,1$.
\end{itemize}

Let us choose   a sequence  of generating involutions $\iota_n:\{k_n,\ldots,k_n+l_n\}\longrightarrow\{k_n,\ldots,k_n+l_n\}$, $n\in \mathbb N$ satisfying $k_0=-1$, $k_1=l_0$, $k_{n+1}=k_n+l_n+2$ for $n\geq1$. We define an involution  $\iota=\bigsqcup_{n=0}^\infty\iota_n:\mathbb Z\longrightarrow\mathbb Z$ as a "limit" construction:
\begin{equation}\label{siota}
  \iota_0,\;\iota_0\sqcup\iota_1,\;\iota_0\sqcup\iota_1\sqcup\iota_2,\, \ldots\, ,\iota_0\sqcup\iota_1\sqcup\iota_2\sqcup\ldots\sqcup\iota_n,\,\ldots\;,
\end{equation}
i.e. we require that $\iota\mid_{\{-n-1,\ldots,k_{n}+l_{n}+1\}}=\iota_0\sqcup\iota_1\sqcup\iota_2\sqcup\ldots\sqcup\iota_{n}$ for each $n\geq1$.
If we assume that $\delta_{k_n}=1$ for all $n$ then it follows immediately from the construction that such defined $\iota$ satisfies (\ref{eqiota}).
\begin{theorem}\label{str}
Let $\wh r_2$, $\wh r_3$, $\wh r_\infty\in\{0,1,\ldots\}\cup\{\infty\}$,  satisfy  $\wh r_2+\wh r_3+\wh r_\infty=\infty$ and $\wh r_\infty\geq1$. Then there is a Neumann subgroup $\wh S<\wh M$ such that $\wh S\setminus M\neq\emptyset$ and $\wh S$ has $(\wh r_2,\wh r_3,\wh r_\infty)$-structure.
\end{theorem}
\begin{proof}
First we define the three generating involutions  with some properties to describe in a moment and the terms of the required sequence of generating involutions will be taken from among those three ones. We require that having chosen a finite sequence of the generating involutions  the next one to be choose delivers some generators  independent of generators brought by previously chosen involutions. In that way we assure that the constructed group will be a free product. Then we have to choose the three generating involutions such that each of them brings appropriate independent generators. We decide to take the following simplest generating involutions (all are denoted by the same symbol $\iota$):
\begin{enumerate}
\item $\iota:\{k\}\longrightarrow\{k\}:\;\iota(k)=k$, $\delta_k=1$;\\ -- delivers a generator of order $2$;\vspace{2mm}
\item $\iota:\{k,k+1\}\longrightarrow\{k,k+1\}:\;\iota(k)=k+1$, $\delta_k=1$;\\ -- delivers a generator of order $3$;\vspace{2mm}
\item $\iota:\{k,k+1,k+2,k+3\}\longrightarrow\{k,k+1,k+2,k+3\}:\;\iota(k)=k+3,\;\iota(k+1)=k+2$, $\delta_k=1$, $\delta_{k+1}=-1$; \\
    -- delivers a generator of infinite order (with the determinant $-1$).
\end{enumerate}
In every case but 1. we have relations
$$\mbox{R($j$)}:\;
\sigma_{k+j}=\sigma_{k+j+1}\sigma_{\iota(k+j+1)+\delta_{k+j+1}},\;\;\;j=0,\ldots,l-1$$
and
$$I(j):\;\sigma_{k+j}\sigma_{\iota(k+j)}=1,\;\;\;j=0,\ldots,l$$
with appropriately taken $l$.

First we show that in each case the generating involution delivers some number of independent generators.\\
\textbf{Case 1.} We have just one generator of order $2$.\\
\textbf{Case 2.} We have $l=1$ and we can drop $I(0)\equiv I(1)$ and the generator $\sigma_{k+1}$. After substituting $\sigma_k^{-1}$ instead $\sigma_{k+1}$ in $R(0)$ we are left with just one generator $\sigma_k$ of order $3$.\\
\textbf{Case 3.} We have $l=3$. Obviously we can drop all the relations $I(0)\equiv I(3)$, $I(1)\equiv I(2)$ and then we can drop the generators $\sigma_{\iota(k+j)}=\sigma_{k+j}^{-1}$ for $j=0,1$. After substituting the left generators into the relations $R(j)$ we see that $R(0)\equiv R(1) \equiv R(2)$, thus we are left with relations $R(0)\equiv\sigma_k=\sigma_{k+1}^2$. Obviously we can drop this relation and   be left with just one generator of infinite order.

We have finished the first step of induction. For each of the three chosen generating involutions we have $det \sigma_k=1$.

Now assume that we have taken  the generating involutions $\iota_i$, $i=0,\ldots, n+1$ from our list and that the involution
$$\iota_0\sqcup\iota_1\sqcup\iota_2\sqcup\ldots\sqcup\iota_{n}$$
brings some number of independent generators. We  have  two new generators $\sigma_{-n-2}$ and $\sigma_{k_{n+1}+l_{n+1}+1}$. Recall that according to our construction both of them have the determinant equal to $1$.  We have to check if the involution $\iota_0\sqcup\iota_1\sqcup\iota_2\sqcup\ldots\sqcup\iota_{n+1}$ have any other independent generators but these delivered by $\iota_0\sqcup\iota_1\sqcup\iota_2\sqcup\ldots\sqcup\iota_{n}$  and $\iota_{n+1}$. The following new relations appear:
$$\begin{array}{ll}
R'(1): & \sigma_{-n-2}=\sigma_{-n-1}\sigma_{k_n+l_n+2},\\&\\
  R'(2):& \sigma_{k_n+l_n+1}=\sigma_{k_n+l_n+2}\sigma_{k_{n+1}+l_{n+1}+1},\\&\\
  R'(3):& \sigma_{k_{n+1}+l_{n+1}}=\sigma_{k_{n+1}+l_{n+1}+1}\sigma_{-n-1}
  \end{array}$$
and $$I':\;\;\sigma_{-n-2}\sigma_{k_{n+1}+l_{n+1}+1}=1.$$
We can drop the generator $\sigma_{k_{n+1}+l_{n+1}+1}=\sigma_{-n-2}^{-1}$ and the relation $I'$. If we substitute the above to the relations $R'(2)$ and $R'(3)$ and use the relations $\sigma_{k_n+l_n+1}\sigma_{-n-1}=1$ and $\sigma_{k_{n+1}+l_{n+1}}\sigma_{k_n+l_n+2}=1$ then we get $R'(1)\equiv R'(2)\equiv R'(3)$. Now we can drop the generator $\sigma_{-n-2}$ and the relation $R'(1)$.
\end{proof}




%

\end{document}